\documentclass[12pt]{amsart}

\usepackage{amsfonts,amssymb,amsmath,amsthm,mathtools,mathabx,bm}
\usepackage{url}
\usepackage{bbm}

\theoremstyle{plain}
\newtheorem{theorem}{Theorem}[section]
\newtheorem*{theorem*}{Theorem}
\newtheorem{lemma}[theorem]{Lemma}

\newtheorem{remark}[theorem]{\bf Remark}

\newtheorem{proposition}[theorem]{\bf Proposition}

\newcommand \IL{\mathcal{L}_{\omega_1 \omega}}
\newcommand \Mod{\rm{Mod}}
\newcommand \tp{\rm{tp}}

\newcommand \bPi{\bm{\Pi}}

\newcommand \KK{\mathcal{K}}

\newcommand \NN{\mathbb{N}}
\newcommand \Seq{\mathbb{N}^{<\mathbb{N}}}
\newcommand \RR{\mathbb{R}}
\newcommand \QQ{\mathbb{Q}}

\newcommand \UU{\mathbb{U}}
\newcommand \UUU{\mathcal{U}}

\newcommand{\cl}[2][]{\overline{#2}^{#1}}

\mathtoolsset{showonlyrefs} 

\title{Isomorphism of almost locally compact Polish metric structures}

\author[M. Malicki]{Maciej Malicki}
\address{Institute of Mathematics, Polish Academy of Sciences, ul. Sniadeckich 8, Warsaw, Poland}
\email{mamalicki@gmail.com}

\keywords{equivalence relations, infinitary continuous logic, almost locally compact structures.}
\subjclass[2020]{Primary 03E15; Secondary 03C66, 03C75}

\begin{document}
\maketitle

\begin{abstract}
A topological space is almost locally compact if it contains a dense locally compact subspace.  We generalize a result from \cite{Ma}, showing that isomorphism on Borel classes of almost locally compact Polish metric structures is always classifiable by countable structures. This allows to remove a gap in the proof presented in \cite{Ma} of a positive answer to a question of Gao and Kechris, who asked whether isometry of locally compact Polish metric spaces is classifiable by countable structures. 
\end{abstract}

\section{Introduction}
One of the main results of \cite{Ma} is Theorem 4.6:

\begin{theorem*}
Let $F$ be a countable fragment, and let $T$ be an $F$-theory all of whose models are locally compact. Then the isomorphism relation on $\Mod(T)$ is classifiable by countable structures.
\end{theorem*}

In this note, we point out that the assumption of local compactness can be relaxed to almost local compactness. Recall that a topological space is called \emph{almost locally compact} if it contains a dense locally compact subspace.  In particular, this generalization fixes a gap in the proof presented in \cite{Ma} of a positive answer to the problem posed by Gao and Kechris whether the isometry relation $\cong$ on the space $\mathcal{LC}$ of locally compact Polish metric spaces is classifiable by countable structures. Here, $\mathcal{LC}$ is understood as subset of the standard Borel space $\KK(\UU)$ of closed subspaces of the Urysohn space $\UU$.  A meaningful positive answer to this question requires finding a standard Borel space $X$, and an equivalence relation $E$ on $X$ such that $\mathcal{LC}$ is a subset of $X$, $E$ restricted to $X$ is $\cong$, and $E$ is classifiable by countable structures. In the proof, a wrong claim was made that $\mathcal{LC}$ itself is Borel in $\KK(\UU)$. In fact, it is complete co-analytic (see \cite{GaKe}). We show that one can take as $X$ the standard Borel space $\mathcal{ALC} \subseteq \KK(\UU)$ of closed almost locally compact subspaces of $\UU$, and the isometry relation on $\mathcal{ALC}$ as $E$.

\section{AE families}
For a Polish metric structure $M$ (see \cite{Ma} for undefined notions and notation used in this note) with metric $d$, we extend $d$ to the set $M^{<\omega}$ of tuples in $M$ by putting, for $m,n < \omega$, $\bar{a} \in M^m$, $\bar{b} \in N^n$,
\begin{equation*}
	d(\bar a, \bar b) = \max \{d(a_i, b_i): i<\min(m,n) \}.
\end{equation*}
In particular, the open ball $B^{M^{<\omega}}_r(\emptyset)$ is equal to $M^{<\omega}$. Also, for $x \in M^n$, the symbols $B_{r}(x)=B_{<r}(x)$, and $B_{\leq r}(x)$ always denote $d$-balls in $M^n$, while for $x \in S_n(T)$, they denote $\partial$-balls in  $S_n(T)$.  
 
Let $F$ be a fixed fragment in signature $L$, and let $\bar{x}$ be a tuple of free variables. An $(-1)$-AE family $P(\bar{x})$ is a formula $\phi(\bar{x})$ in $F$. Provided that $\gamma$-AE families have been defined for $\gamma<\beta$, where $\beta=0$ or $\beta$ is a limit ordinal, a $\beta$-AE family $P(\bar{x})$ is a collection of $\gamma$-AE families $p_k(\bar{x})$, $k \in \NN$, $\gamma<\beta$, a $(\beta+1)$-AE family $P(\bar{x})$ is a collection of $\gamma$-AE families $p_{k,l}(\bar{x}_{k,l})$, $\gamma<\beta$, $k,l \in \NN$, $\bar{x} \subseteq \bar{x}_{k,l}$, and a $(\beta+n)$-AE family $P(\bar{x})$, $2 \leq n < \omega$, is a collection of $(\beta+n-2)$-AE families $p_{k,l}(\bar{x}_{k,l})$, $k,l \in \NN$, $\bar{x} \subseteq \bar{x}_{k,l}$. Moreover, every $\alpha$-AE family $P(\bar{x})=\{p_{k,l}(\bar{x}_{k,l})\}$, $\alpha \geq 1$, comes equipped with a fixed $u_P \geq 0$ such that $u_P \geq u_{p_{k,l}}$, $k,l \in \NN$.

Let $\bar{a}$ be an $n$-tuple in a Polish metric structure $M$. We say that $\bar{a}$ realizes a $(-1)$-AE family $P(\bar{x})=\phi(\bar{a})$ if $\phi^M(\bar{a})=0$, and $\bar{a}$ realizes a $\beta$-AE family $P(\bar{x})$, where $\beta=0$ or $\beta$ is a limit ordinal, if it realizes every $p(\bar{x}) \in P(\bar{x})$. Finally, $\bar{a}$ realizes a $(\beta+n)$-AE family $P(\bar{x})=\{p_{k,l}(\bar{x}_{k,l})\}$, $1 \leq n< \omega$, if it holds in $M$ that 
\[ \forall \bar{b} \in B^{M^{<\omega}}_{u_P}(\bar{a}) \forall v>0 \forall k \exists \bar{c} \in B^{M^{<\omega}}_{v}(\bar{b}) \exists l  \, (\bar{c} \mbox{ realizes } p_{k,l}(\bar{x}_{k,l}) \mbox{ in } M ). \]
If $\emptyset$ in $M$ realizes $P(\emptyset)$, we say that $M$ models $P$.

Let 
\[	\rho_M(\bar{a}) = \sup \{r \in \RR : \cl{B_r(\bar{a})} \text{ is compact} \};\]
we write $\rho(\bar{a})$ instead of $\rho_M(\bar{a})$, when $M$ is clear from the context. Recall that every $M \in \Mod(L)$ is a code for a Polish metric structure, also denoted by $M$, that consists of $\Seq$ and relations on $\Seq$. Thus, we always assume that Polish metric structures contain $\Seq$ as a (tail)-dense subset. Now, Remark 3.1 in \cite{Ma} can be obviously strengthened as follows.

\begin{remark}
\label{re:Smallv}
Note that in order to verify that $\bar{a}$ realizes $P(\bar{x})$, it suffices to check that the above condition holds for some dense set of $\bar{b} \in B^{M^{<\omega}}_{u_P}(\bar{a})$, and for all sufficiently small $v>0$. In particular, one can assume that $\bar{b} \in \Seq$, and, if $M$ is almost locally compact, that $\rho_M(\bar{b})>0$.   
\end{remark}

The crucial application of AE families in \cite{Ma} is the following result (Corollary 3.3):

\begin{theorem}
	\label{co:IsoAEAlpha}
	Let $F$ be fragment in signature $L$, and let $1 \leq \alpha<\omega_1$. Suppose that $[M] \in \bPi^0_{1+\alpha}(t_F)$ for some $M \in \Mod(L)$. There exists an $\alpha$-AE family $P_M$ such that
	\[  [M]= \{ N \in \Mod(L) : N \mbox{ models } P_M \}. \]
\end{theorem}

\section{Almost locally compact structures}
Let $F$ be a fragment, let $T$ be an $F$-theory, and let $M \in \Mod(T)$ be almost locally compact. For $n \in \NN$, define 
\[ \Theta_n(M)=\{\tp_F(\bar{a}) \in S_n(T): \bar{a} \in M^n\}.\]
The following fact, proved in \cite{HaMaTs} for locally compact $M$ (Lemma 6.2), holds with the same proof for almost locally compact $M$.

\begin{lemma}
	\label{l:loc-cpct-type-space}
	Let $\Phi \colon (M^n, d) \to (S_n(T), \partial)$ be defined by $\Phi(\bar a) = \tp (\bar a)$. Then the following hold:
	\begin{enumerate}
		\item \label{i:p:loc-cpct-type-space:1} $\Phi$ is a contraction for the metrics $d$ on $M^n$ and $\partial$ on $S_n(T)$.
		\item \label{i:p:loc-cpct-type-space:2} If $r < \rho(\bar a)$, then $\Phi(B_{\leq r}(\bar a)) = B_{\leq r}(\Phi(a))$. In particular, \\ $B_{\leq r}(\tp (\bar a)) \subseteq \Theta_n(M)$, $B_{\leq r}(\tp (\bar a))$ is $\partial$-compact, and $\tau_n$- and $\partial$-topology coincide on $B_{\leq r}(\tp (\bar a))$.
		\item \label{i:p:loc-cpct-type-space:3} If $r \leq \rho(\bar a)$, then $\Phi(B_r(\bar a)) = B_r(\Phi(\bar a))$.
	\end{enumerate}
\end{lemma}

For each logic topology $\tau_n$ on $S_n(T)$ fix a countable basis $\UUU_n=\{U_{l,n}\}$ containing $\emptyset$ and the whole space, and put $\UUU=\bigcup_n \UUU_n$. Let $U \in \UUU_n$, $\epsilon>0$, and let $\bar{a}$ be an $n$-tuple in $M$ with $\rho(\bar{a})>0$. We say that $(U,\epsilon)$ is \emph{$\bar{a}$-good} in $M$ if
\begin{itemize}
	\item $\tp(\bar{a}) \in U$,
	\item $2\epsilon<\rho(\bar{a})$,
	\item there is $\delta>0$ such that $U \cap B_{2\epsilon}(\tp(\bar{a})) \subseteq B_{\epsilon-\delta}(\tp(\bar{a}))$. 
\end{itemize}

\begin{remark}
\label{re:good}
The following observations easily follow from the fact that $\partial$- and $\tau$- topologies coincide on compact subsets of $(S_n(T),\partial)$.
\begin{enumerate}
\item For every $\delta>0$ there exist $U \in \UUU$ and $0<\epsilon<\delta$ such that $(U,\epsilon)$ is $\bar{a}$-good,
\item if $(U,\epsilon)$ is $\bar{a}$-good, then $$\cl[\tau]{B_\epsilon(\tp(\bar{a})) \cap U } \subseteq \Theta_{|\bar{a}|}(M),$$
\item if $(U,\epsilon)$ is $\bar{a}$-good, there is $\delta>0$ such that $d(\bar{a},\bar{a}')<\delta$ implies that  $(U,\epsilon)$ is $\bar{a}'$-good, and $$U \cap B_{\epsilon}(\tp(\bar{a}))=U \cap B_{\epsilon}(\tp(\bar{a}')).$$
\end{enumerate}
\end{remark}
Now, for $\bar{a} \in \Seq$, $U \in \UUU_n$, and $\epsilon \in \QQ^+$, define
\[ T^0_{U,\epsilon}(\bar{a})=\cl[\tau]{B_\epsilon(\tp(\bar{a})) \cap U }, \]
if $(U,\epsilon)$ is $\bar{a}$-good,
\[ T^0_{U,\epsilon}(\bar{a})=\emptyset, \]
otherwise, and
\begin{multline}
T^{\alpha}_{U,\epsilon}(\bar{a})=\{ T^\beta_{U',\epsilon'}(\bar{a}'): \beta<\alpha, |\bar{a}'| \geq |\bar{a}|, U' \in \UUU_{|\bar{a}'|}, U' \upharpoonright |\bar{a}| \subseteq U, \epsilon' \in \QQ^+, \epsilon' \leq \epsilon \} 
\end{multline}
for $\alpha>0$. Also, for $u \in \QQ^+$, put
\begin{multline}
T^\alpha_u(\bar{a})=\{ T^\beta_{U,v}(\bar{b}): \beta<\alpha, \bar{b} \in B^{M^{<\omega}}_u(\bar{a}), |\bar{b}| \geq |\bar{a}|, U \in \UUU_{|\bar{b}|}, v \in \QQ^+ \}, 
\end{multline}

\[ T^\alpha(M)=T^\alpha_1(\emptyset).\]
Using Remark \ref{re:good}(3), it is straightforward to observe that

\begin{remark}
\label{re:CongSameT}
$M \cong N$ implies that $T^\alpha(M)=T^\alpha(N)$ for all $\alpha \geq 1$.
\end{remark}

The following result, proved in \cite{Ma} for locally compact $M, N$ (Proposition 4.4), holds with the same proof for almost locally compact $M, N$. We present the proof for the sake of completeness.

\begin{proposition}
\label{pr:AET}
Let $F$ be a fragment, and let $T$ be an $F$-theory. Assume that $M,N \in Mod(T)$ are almost locally compact, and  $T^\alpha_{u}(\bar{a})=T^\alpha_{u'}(\bar{a}')$ for some tuples $\bar{a}$, $\bar{a}'$ in $M$, $N$, respectively. Then every $\alpha$-AE family $P(\bar{x})$ with $u_P \leq u$ realized by $\bar{a}'$, is also realized by $\bar{a}$. 
\end{proposition}

\begin{proof}
Suppose that $T^1_{u}(\bar{a})=T^1_{u'}(\bar{a}')$, and fix a $1$-AE family $P(\bar{x})=\{p_{k,l}(x_{k,l})\}$ realized by $\bar{a}'$, and with $u_P \leq u$. Fix $\bar{b} \in B^{M^{<\omega}}_{u_P}(\bar{a})$, $v>0$, and $k \in \NN$. By Remarks \ref{re:Smallv} and \ref{re:good}(1), we can assume that $v<\rho(\bar{b})$, and there is $U \in \UUU$ such that $(U,v/2)$ is $\bar{b}$-good. Find $\bar{b}' \in  B^{N^{<\omega}}_{u_P}(\bar{a}')$, $U' \in \UUU$, and $v'>0$ such that $T^0_{U,v/2}(\bar{b})=T^0_{U',v'}(\bar{b}')$. As $\bar{a}'$ realizes $P(\bar{x})$, there is $\bar{c}' \in B^{N^{<\omega}}_{v/2}(\bar{b}')$, and $l$ such that $p_{k,l}^N(\bar{c}')=0$. But then 
\[ \inf_{\bar{x}} [(d(\bar{b},\bar{x}) \dotdiv v/2) \vee p_{k,l}(\bar{x})] \in \tp(\bar{b}'), \]
and, by Remark \ref{re:good}(2), there is $\bar{d} \in  B^N_{v/2}(\bar{b})$ with $\tp(\bar{d})=\tp(\bar{b}')$. By the compactness of $B^M_{v/2}(\bar{d})$, there is $\bar{c} \in B^M_{v/2}(\bar{d})$ such that $p_{k,l}^M(\bar{c})=0$. Clearly, $\bar{c} \in B^{M^{<\omega}}_v(\bar{b})$. As $\bar{b}$, $v$ and $k$ were arbitrary, this shows that $\bar{a}$ realizes $P(\bar{x})$.

Suppose now that $T^2_{u}(\bar{a})=T^2_{u'}(\bar{a}')$, and let $P(\bar{x})$ be a $2$-AE family realized by $\bar{a}'$, and with $u_P \leq u$. Fix $\bar{b} \in  B^{M^{<\omega}}_{u_P}(\bar{a})$, $v>0$, $k \in \NN$, and $U \in \UUU$ such that $(U,v/2)$ is $\bar{b}$-good. Fix $\bar{b}' \in B^{N^{<\omega}}_{u_P}(\bar{a}')$, $U' \in \UUU$ and $v'>0$ such that $T^1_{U,v/2}(\bar{b})=T^1_{U',v'}(\bar{b}')$. As $\bar{a}'$ realizes $P(\bar{x})$, there is $l$, and $\bar{c}' \in B^{N^{<\omega}}_{v/2}(\bar{b}')$ such that $\tp(\bar{c}') \upharpoonright |\bar{b}'| \in U'$, and  $\bar{c}'$ realizes $p_{k,l}(\bar{x}_{k,l})$. Fix $V,V' \in \UUU$, $0<w,w' \leq v/2$, and $\bar{d} \in B^{M^{<\omega}}_{v/2}(\bar{b})$ such that $(V',w')$ is $\bar{c}'$-good, and $T^0_{V,w}(\bar{d})=T^0_{V',w'}(\bar{c}')$. Then there is $\bar{c} \in B^{M^{<\omega}}_{w}(\bar{d})$ with $\tp(\bar{c})=\tp(\bar{c'})$, i.e., $\bar{c} \in B^{M^{<\omega}}_v(\bar{b})$, and $\bar{c}$ realizes $p_{k,l}(\bar{x}_{k,l})$.
	
For $\alpha>1$, this is an easy induction.
\end{proof}

Proposition \ref{pr:AET} combined with Theorem \ref{co:IsoAEAlpha} immediately gives a generalization of Theorem 4.5 from \cite{Ma}:

\begin{theorem}
\label{th:isoTalpha}
Let $F$ be a fragment, and let $T$ be an $F$-theory all of whose models are almost locally compact. Suppose that $[M] \in \Pi^0_{1+\alpha}(t_F)$, $\alpha \geq 1$, for some $M \in \Mod(T)$. Then $$[M]=\{N \in \Mod(T): T^{\alpha}(N)=T^{\alpha}(M)\}.$$
\end{theorem}


Now we are in a position to prove a generalization of Theorem 4.6 from \cite{Ma}, essentially with the same proof. The only place, where local compactness is assumed, is an application of Lemma 6.4 from \cite{HaMaTs}. The reader can easily verify that in fact this assumption is not needed there.

\begin{theorem}
\label{th:CtbleModels}
Let $F$ be a fragment, and let $T$ be an $F$-theory all of whose models are almost locally compact. Then $\cong_T$ is classifiable by countable structures.
\end{theorem}

\begin{proof}
First, for a given $M \in \Mod(T)$, we construct a countable structure $C_M$, essentially, as in the proof of \cite{Hj}[Lemma 6.30]. Its universe consists of elements $x$ of the form $$x=(\cl[\tau]{B_\epsilon(\tp(\bar{a})) \cap U },|\bar{a}|,U,\epsilon),$$ where $\bar{a} \in \NN^{<\NN}$, $U \in \UUU_{|\bar{a}|}$, $\epsilon \in \QQ^+$, and $(U,\epsilon)$ is $\bar{a}$-good. The relevant information carried by these objects is recorded with an aid of the relations $O_l$, $R_{k,l,\delta}$, $k,l \in \NN$, $\delta \in \QQ^+$, and $E$, defined, for $x=(\cl[\tau]{B_\epsilon(\tp(\bar{a})) \cap U },|\bar{a}|,U,\epsilon)$, $x'=(\cl[\tau]{B_{\epsilon'}(\tp(\bar{a}')) \cap U' },|\bar{a}'|,U',\epsilon')$, as follows:

\begin{itemize}
	\item $O_l(x)$ iff $U_{l,|\bar{a}|} \cap \cl[\tau]{B_\epsilon(\tp(\bar{a})) \cap U }=\emptyset$,
	\item $R_{k,l,\delta}(x)$ iff $k=|\bar{a}|$, $U=U_{l,k}$, $\delta=\epsilon$,
	\item $x E x'$ iff $|\bar{a}'| \geq |\bar{a}|$, $U' \upharpoonright |\bar{a}| \subseteq U$,  $\epsilon' \leq \epsilon$.
\end{itemize}

By Remark \ref{re:good}(3), $M \cong N$ implies that $C_M=C_N$. On the other hand, as the relations $O_l$ record complements of sets $\cl[\tau]{B_\epsilon(\tp(\bar{a})) \cap U }$ in $S_{|\bar{a}|}(T)$, we have that $$\pi((\cl[\tau]{B_\epsilon(\tp(\bar{a})) \cap U },|\bar{a}|,U,\epsilon))=(\cl[\tau]{B_\epsilon(\tp(\bar{a})) \cap U },|\bar{a}|,U,\epsilon)$$ for any isomorphim $\pi:C_M \rightarrow C_N$.
It obviously follows that $T^0_{U,\epsilon}(\bar{a})=T^0_{U,\epsilon}(\bar{a'})$, if $\pi(x)=x'$. And $E$ warranties that $T^\alpha_{U,\epsilon}(\bar{a})=T^\alpha_{U,\epsilon}(\bar{a'})$ for $\alpha>0$. In particular, $T^\alpha(M)=T^\alpha(N)$ for $\alpha<\omega_1$.

It is not hard to construct a Borel mapping $\Mod(T) \rightarrow 2^\NN$, $M \mapsto D_M$, so that $D_M$ codes a countable structure isomorphic to $C_M$. First, by \cite[Lemma 6.4]{HaMaTs} (note that local compactness of $M$ is not relevant in the proof of Lemma 6.4(i)), the mappings
\[ \Mod(T) \times \NN^n \rightarrow \RR, \ (M,\bar{a}) \mapsto \rho_M(\bar{a}), \]
\[ \Mod(T) \times \NN^n \times \QQ^+ \rightarrow \mathcal{K}(S_n(T)), \]
\[ (M,\bar{a},r) \mapsto B_{\leq r}(\tp(\bar{a})) \mbox{ if } r<\rho_M (\bar{a}) \mbox{ and } \emptyset, \mbox{ otherwise, } \]
where $\mathcal{K}(X)$ is the standard Borel space of closed subsets of $X$, are Borel. Therefore the relation  ``$(U,\epsilon)$ is $\bar{a}$-good in $M$'', regarded as a subset of $\Mod(T) \times \NN^{<\NN} \times \UUU \times \QQ^+$, is also Borel. This gives rise to a Borel enumeration $e: \NN \rightarrow (\bigsqcup_n \KK(S_n(T))) \times \UUU \times \QQ^+$ of the universe of $C_M$. Using this $e$, we can easily construct the required Borel mapping $M \mapsto D_M$. 

By \cite[Corollary 5.6]{BeDoNiTs}, for every $M \in \Mod(T)$, the isomorphism class $[M]$ is Borel, i.e., $[M] \in \bPi^0_\alpha(t_F)$ for some $\alpha<\omega_1$. Hence, Theorem \ref{th:isoTalpha} implies that $M \cong N$ iff $D_M \cong D_N$. 
\end{proof}

Let $\UU$ be the Urysohn space, and let $\mathcal{ALC} \subseteq  \KK(\UU)$ be the set of all closed almost locally compact subspaces of $\UU$.

\begin{proposition}
	$\mathcal{ALC}$ is a Borel subset of $\KK(\UU)$.
\end{proposition} 

\begin{proof}
	By the Kuratowski--Ryll-Nardzewski theorem, there exists a Borel $\phi: \KK(\UU) \rightarrow \UU^\NN$ such that $\phi(K)$ is a dense sequence in $K$. It is straightforward to verify that each mapping $K \mapsto \rho_K(\phi(K)_n)$ is Borel (see the proof of Lemma 6.4 in \cite{HaMaTs} for more details). Thus, the condition for $K$ being almost locally compact can be formulated as
	\[ \forall m \forall q \in \QQ^+ \exists n (d(\phi(K)_m,\phi(K)_n)<q \mbox{ and } \rho_K(\phi(K)_n)>0),   \]
	which is clearly Borel.
\end{proof}

\begin{theorem}
\label{th:Isometry}
Isometry of locally compact Polish metric spaces is classifiable by countable structures.
\end{theorem}

\begin{proof}
Every almost locally compact Polish metric space $K$ can be coded as $M_K \in \Mod(L)$ with the trivial signature $L$, and metric bounded by $1$. Simply, pick a countable tail-dense subset of $K$, and replace the original metric $d$ with the metric $1/(1+d)$ which does not change the isometry relation. Actually, for $\mathcal{ALC} \subseteq \mathcal{K}(\UU)$, the coding $\mathcal{ALC} \rightarrow \Mod(L)$, $K \mapsto M_K$, can be defined in a Borel way: the Kuratowski--Ryll-Nardzewski theorem yields a Borel function $f:\mathcal{K}(\UU) \rightarrow \UU^\NN$ such that $f(K)$ is a tail-dense sequence in $K$. As the signature $L$ is trivial, the isomorphism relation is just the isometry relation. Moreover, the property of being almost locally compact can be expressed as a sentence in $\IL(L)$, so the set of all possible codes of almost locally compact Polish metric spaces is of the form $\Mod(T)$. By Theorem \ref{th:CtbleModels} and \cite[Theorem 13.1.2]{Gao}, the isometry relation on $\mathcal{ALC}$ is classifiable by countable structures.

Thus, the isometry relation on the (complete co-analytic) subspace $\mathcal{LC} \subseteq \mathcal{ALC}$ of locally compact Polish metric spaces is classifiable by countable structures.
\end{proof}

\end{document}